\documentclass{tbimc}

\usepackage[shortlabels]{enumitem}
\usepackage{mathtools,bbm}

\newcommand{\F}{\mathcal F}
\newcommand{\FF}{\mathbb F}
\newcommand{\E}{\mathsf E}
\newcommand{\R}{\mathbb R}
\newcommand{\N}{\mathbb N}
\newcommand{\B}{\mathcal B}
\newcommand{\Beta}{\mathrm B}
\newcommand{\eps}{\varepsilon}
\newcommand{\oD}{\overline D}
\newcommand{\Der}{\mathcal D}
\newcommand{\ind}{\mathbbm1}
\DeclareMathOperator{\Div}{div}
\newcommand*{\set}[1]{\left\{#1\right\}}
\newcommand*{\abs}[1]{\left\lvert#1\right\rvert}
\newcommand*{\norm}[1]{\left\lVert#1\right\rVert}

\newtheorem{theorem}{Theorem}[section]
\newtheorem{lemma}{Lemma}[section]
\newtheorem{proposition}{Proposition}[section]
\theoremstyle{definition}
\newtheorem{definition}{Definition}[section]
\theoremstyle{remark}
\newtheorem{remark}{Remark}[section]

\allowdisplaybreaks

\begin{document}
\selectlanguage{english}

\title[Mild solution to stochastic heat equation]{Existence and uniqueness of mild solution\\
to stochastic heat equation\\ with white and fractional noises}

\author{Yu. Mishura}
\address{Department of Probability Theory, Statistics and Actuarial Mathematics, Taras Shev\-chenko National University of Kyiv,
64 Volodymyrska, 01601 Kyiv, Ukraine} \email{myus@univ.kiev.ua}

\author{K. Ralchenko}
\address{Department of Probability Theory, Statistics and Actuarial Mathematics, Taras Shev\-chenko National University of Kyiv,
64 Volodymyrska, 01601 Kyiv, Ukraine}
\email{k.ralchenko@gmail.com}

\author{G. Shevchenko}
\address{Department of Probability Theory, Statistics and Actuarial Mathematics, Taras Shev\-chenko National University of Kyiv,
64 Volodymyrska, 01601 Kyiv, Ukraine}
\email{zhora@univ.kiev.ua}

\date{}

\begin{abstract}
We prove the existence and uniqueness of a mild solution for a class of non-autonomous parabolic mixed stochastic partial
differential equations defined on a bounded open subset $D \subset \R^d$ and involving standard and fractional $L^2(D)$-valued Brownian motions.
We assume that the coefficients are homogeneous, Lipschitz continuous and the coefficient at the fractional Brownian motion is an affine function.
\end{abstract}

\maketitle

\section{Introduction}

Let $(\Omega,\F,\Prob)$ be a complete probability space.
For a fixed $T>0$ let $\FF=\set{\F}_{t\in[0,T]}$ be a filtration satisfying the standard assumptions.
Let $\beta\in(0,1)$ be fixed throughout the paper.
Assume that $D\subset\R^d$ is a bounded domain with boundary $\partial D$ of class $C^{2+\beta}$.

We consider the following stochastic partial differential equation with boundary conditions
\begin{align}
du(x,t) &= \Bigl(\Div\bigl(k(x,t)\nabla u(x,t)\bigr)+f\bigl(u(x,t)\bigr)\Bigr)dt
+ g\bigl(u(x,t)\bigr) W(x,dt) \notag
\\
&\quad + h\bigl(u(x,t)\bigr) W^H(x,dt),
\qquad (x,t) \in D\times(0,T], \label{eq:spde}
\\
u(x,0) &= \varphi(x),
\qquad x \in D, \label{eq:init}
\\
\frac{\partial u(x,t)}{\partial n(k)}&=0,
\qquad (x,t)\in \partial D \times (0,T]. \label{eq:boundary}
\end{align}
Here
$W$ is an $L^2(D)$-valued Wiener process and $W^H$ is an $L^2(D)$-valued fractional Brownian motion with the Hurst index $H\in(1/2,1)$.
Furthermore, $k=\set{k_{i,j}}\colon D\to \R^{d\times d}$ is a matrix-valued field, consequently,
\[
\Div\bigl(k(x,t)\nabla u(x,t)\bigr)=\sum_{i,j=1}^d \frac{\partial}{\partial x_i} \left(k_{i,j}(x,t) \frac{\partial}{\partial x_j}u(x,t)\right);
\]
$n(k)(x)\coloneqq k(x,t)n(x)$ denotes the conormal vector-field, and the relation \eqref{eq:boundary} stands for conormal derivative of $u$ relative to  $k$, that is
\[
\frac{\partial u(x,t)}{\partial n(k)}
=\sum_{i,j=1}^d k_{i,j}(x,t) n_i(x) \frac{\partial}{\partial x_j}u(x,t),
\]
where $n(x)\in\R^d$ is an outer normal vector to $\partial D$.
We are interested in the existence and uniqueness of a mild solution to \eqref{eq:spde}--\eqref{eq:boundary}.
The precise statement of the problem and the definition of a mild solution will be given in Sections \ref{sec:prelim} and \ref{sec:mild}.

In the pure Wiener case, where $h=0$, the problem was investigated by Sanz-Sol\'e and Vuillermot \cite{SSV03}, who introduced three different notions for solutions (namely, variation solutions of the first and the second kind and mild solutions)
and showed their indistinguishability.
Also, they proved the existence, the uniqueness and the pointwise boundedness of the moments along with the spatial Sobolev regularity of such solutions.
Later their results for mild solutions were improved in several directions by Veraar \cite[Sec.~8]{Veraar10}.
In particular, he proved the existence and uniqueness of a mild solution in the case of random coefficients, depending also on $(x,t)$.

The pure ``fractional'' case, where $g = 0$, was studied in \cite{NV06} and \cite{SSV09}.
In \cite{NV06} the existence, uniqueness and indistinguishability of two types of variational solutions were proved, assuming that the coefficients $f$ and $h$ are Lipschitz continuous and the derivative of $h$ is H\"older continuous.
In \cite{SSV09} the authors proved the existence of a mild solution and established its relation with the variational solution of type II from \cite{NV06} and the H\"older continuity of its sample paths.
When $h$ is an affine function, they also proved the uniqueness of a mild solution and the indistinguishability of mild and variational solutions.

In this paper we will consider mild formulation of the stochastic heat equation \eqref{eq:spde} with white and fractional noises:
\begin{equation}
\begin{aligned}
u(\cdot,t) &= U(t,0)\varphi(\cdot)
+ \int_0^t U(t,s)\big(f(u(\cdot,s))\,ds \\ & +
g(u(\cdot,s))dB(\cdot,s)
+ h(u(\cdot,s))dB^H(\cdot, s),
\end{aligned}
\label{eq:mild1}
\end{equation}
where $U(t,s) = \exp\set{\int_s^t A_u du}$ is the evolution family on $L^2(D)$ corresponding to the elliptic operator $A_t\psi(\cdot) = \operatorname{div}\big(k(\cdot,t) \nabla \psi(\cdot)\big)$ with Neumann boundary conditions \eqref{eq:boundary}; the precise formulations will be given later. We show that equation \eqref{eq:spde} has a unique mild solution.
The conditions on the coefficients are similar to those of
\cite{SSV03} and \cite{SSV09}.
In particular, the functions $f$ and $g$ are assumed to be Lipschitz continuous, and $h$ is assumed to be an affine function.
In order to analyze the equation with two different noises (and, consequently, with two different types of stochastic integrals), we replace the fractional Brownian motion by a smooth process, transforming the equation \eqref{eq:spde} into a stochastic partial differential equation with random drift driven by a Wiener process.
This approach was developed for ordinary stochastic differential equations involving both Wiener process and fractional Brownian motion in the article \cite{MSh12}.
In \cite{Sh14} it was applied to mixed stochastic delay equations.

We organize this article in the following way.
In Section~\ref{sec:prelim}, we formulate the assumptions, define $L^2(D)$-valued Wiener and fractional Brownian processes and introduce the corresponding stochastic integrals.
Also, this section contains some properties of Green's function associated with our equation.
In Section~\ref{sec:mild} we define a mild solution and prove its existence and uniqueness.
In Appendix, we collect auxiliary estimates for solutions that are used for the proving of main result.

\section{Preliminaries}
\label{sec:prelim}

\subsection{Assumptions on the coefficients and on the initial value}
\begin{enumerate}[label=\bf(A\arabic*),ref=\rm(A\arabic*)]
\item\label{(A1)}
Assumptions on $k$ and $n$:
\begin{enumerate}[(i)]
\item
$k_{i,j}=k_{j,i}$ for all $i,j=1,\ldots,d$;
\item
$k_{i,j} \in C^{\beta,\beta'}(\oD\times[0,T])$
for some $\beta'\in(\frac12,1]$ and for all $i,j=1,\ldots,d$;
\item
$\frac{\partial}{\partial x_l} k_{i,j} \in C^{\beta,\beta/2}(\oD \times [0,T])$ for all $i,j,l=1,\ldots,d$;
\item
there exists $\underline k>0$ such that
\[
\sum_{i,j=1}^d k_{i,j}(x,t)q_iq_j \ge \underline k \abs{q}^2,
\]
for all $x\in\oD$, $t\in[0,T]$, $q\in\R^d$
(here $\abs{\,\cdot\,}$ denotes the Euclidean norm in~$\R^d$);
\item
$\displaystyle (x,t)\mapsto\sum_{i=1}^dk_{i,j}(x,t)n_i(x) \in C^{1+\beta,(1+\beta)/2}(\partial D\times[0,T])$ for each $j$;
\item
the conormal vector-field
$(x,t)\mapsto n(k)(x,t)= k(x,t)n(x)$
is outward pointing, nowhere tangent to $\partial D$ for every $t$.
\end{enumerate}

\item\label{(A4)}
The initial condition satisfies $\varphi\in C^{2+\beta}(\oD)$ and the conormal boundary condition \eqref{eq:boundary} relative to $k$.

\item\label{(A2)}
$f,g\colon\R\to\R$ are Lipschitz continuous functions.

\item\label{(A3)}
$h\colon\R\to\R$ is an affine function.

\end{enumerate}

\subsection{Norms and spaces}
Let $\norm{\,\cdot\,}_2$ and $\norm{\,\cdot\,}_\infty$ be the norms in $L^2(D)$ and $L^\infty(D)$ respectively.
For $\alpha\in(0,1)$ denote by $\B^{\alpha,2}\left (0,T;L^2(D)\right)$ the Banach space of Lebesgue-measurable mappings $u\colon[0,T]\to L^2(D)$ endowed with the norm
\[
\norm{u}_{\alpha,2,T}^2 := \left(\sup_{t\in[0,T]}\norm{u(t)}_2\right)^2
+ \int_0^T \left( \int_0^t \frac{\norm{u(t)-u(s)}_2}{(t-s)^{\alpha+1}}\,ds \right )^2 dt < \infty.
\]
Denote also for $u\colon[0,T]\to L^2(D)$
\[
\norm{u}_{\alpha,1,T} := \int_0^T\left(\frac{\norm{u(t)}_2}{t^\alpha}
+  \int_0^t \frac{\norm{u(t)-u(s)}_2}{(t-s)^{\alpha+1}}\,ds \right ) dt.
\]


For $f\colon[0,T]\to\R$ and $\alpha\in(0,1)$ define a seminorm
\[
\norm{f}_{\alpha,0;t} = \sup_{0\le u< v < t}\left(\frac{\abs{f(v)-f(u)}}{(v-u)^{1-\alpha}} + \int_u^v \frac{\abs{f(u)-f(z)}}{(z-u)^{2-\alpha}}\,dz \right).
\]

\subsection{Green's function}
Let $G\colon\set{(x,t,y,s):0\le s<t\le T,x,y\in\oD} \to \R$ be the parabolic Green's function associated with the principal part of \eqref{eq:spde}.
It is known from \cite{EI70,Eidelman98} that under assumptions \ref{(A1)} and \ref{(A4)}
$G$ is a continuous function, twice continuously differentiable in $x$, once continuously differentiable in $t$.
For every $(y,s)\in D\times(0,T]$ it is a classical solution to the linear initial-boundary value problem
\begin{align*}
\partial_t G(x,t;y,s) &= \Div\bigl(k(x,t)\nabla_x G(x,t;y,s)\bigr),
\qquad (x,t) \in D\times(0,T], 
\\
\frac{\partial G(x,t;y,s)}{\partial n(k)}&=0,
\qquad (x,t)\in \partial D \times (0,T],
\\
\int_D G(\cdot,s;y,s) \varphi(y)\,dy &:= \lim_{t\downarrow s} \int_D G(\cdot,t;y,s) \varphi(y)\,dy = \varphi(\cdot).
\end{align*}
Moreover, $G$ satisfies the heat kernel estimates
\[
\abs{\partial_x^\gamma \partial_t^\delta G(x,t;y,s)}
\le C(t-s)^{-\frac12(d+\abs{\gamma}+2\delta)} \exp\set{-C\frac{\abs{x-y}^2}{t-s}}
\]
for $\gamma=(\gamma_1,\dots,\gamma_d)$,
$\gamma_1,\dots,\gamma_d,\delta\in\N\cup\set{0}$, and
$\abs{\gamma} + 2\delta \le 2$ with
$\abs{\gamma} = \sum_{j=1}^d\gamma_j$.
In particular, for $\abs{\gamma}=\delta=0$, we have
\begin{equation}\label{eq:gaus}
\abs{G(x,t;y,s)}
\le C(t-s)^{-\frac d2} \exp\set{-C\frac{\abs{x-y}^2}{t-s}}.
\end{equation}
We shall refer to \eqref{eq:gaus} as the \emph{Gaussian property} of $G$.

The evolution family corresponding to the operator $A_t\psi(\cdot) = \operatorname{div}\big(k(\cdot,t) \nabla \psi(\cdot)\big)$ with Neumann boundary conditions \eqref{eq:boundary} is defined as
\begin{equation}\label{eq:U(t,s)}
U(t,s)\psi(x) = \int_{D} G(x,t;y,s) \psi(y)dy.
\end{equation}
From the Gaussian estimates \eqref{eq:gaus} it can be easily shown
that the family $U(t,s)$ is of contractive type on $L^2(D)$. Indeed,
denoting by $\zeta$ a standard Gaussian vector in $\R^d$, we have from
\eqref{eq:gaus} for $\psi\in L^2(D)$
$$
|U(t,s)\psi(x)| \le C\E \abs{\psi(x + c\sqrt{t-s}\zeta )\ind_{D}(x + c\sqrt{t-s}\zeta)}
$$
with some $c>0$, whence
\begin{align*}
&\norm{U(t,s)\psi}^2_{2}  \le C\int_{D} \big(\E \abs{\psi(x +
c\sqrt{t-s}\zeta )} \ind_{D}(x + c\sqrt{t-s}\zeta)\big)^2 dx \\
&\le   C\E \int_{D}  \psi(x + c\sqrt{t-s}\zeta )^2\ind_{D}(x + c\sqrt{t-s}\zeta)
dx \le C\E \norm{\psi}^2_{2}.
\end{align*}

In the proposition below we collect other useful estimates for $G$,
see Eqs. (3.4), (3.5) and (3.36) in \cite{SSV09}.

\begin{proposition}[\cite{SSV09}]
Under assumption \ref{(A1)},
for all $x,y\in D$ and $\delta\in(\frac{d}{d+2},1)$,  $G$ satisfies the following inequalities.
\begin{enumerate}[(i)]
\item
For all $0<r<v<t<T$ and some $t^*\in(r,v)$,
\begin{equation}\label{eq:deltaG1}
\abs{G(x,t;y,v)-G(x,t;y,r)} \le (t-v)^{-\delta} (v-r)^{\delta} (t-t^*)^{-d/2} \exp\set{-C\frac{\abs{x-y}^2}{t-t^*}}.
\end{equation}
\item
For all $0<v<s<t<T$ and some $v^*\in(s,t)$,
\begin{equation}\label{eq:deltaG}
\abs{G(x,t;y,v)-G(x,s;y,v)} \le (t-s)^\delta (s-v)^{-\delta} (v^*-v)^{-d/2} \exp\set{-C\frac{\abs{x-y}^2}{v^*-v}}.
\end{equation}
\item
For all  $0<r<v<s<t<T$ and some $v^*, r^* \in (s,t)$,
\begin{align}
\MoveEqLeft
\abs{G(x,t;y,v)-G(x,s;y,v)-G(x,t;y,r)-G(x,s;y,r)}
\notag
\\
&\le (t-s)^\delta (s-v)^{-1} (v-r)^{1-\delta}
\notag
\\
&\quad\times \left(
(v^*-v)^{-d/2} \exp\set{-C\frac{\abs{x-y}^2}{v^*-v}}
+(r^*-r)^{-d/2} \exp\set{-C\frac{\abs{x-y}^2}{r^*-r}}\right).
\label{eq:deltaG2}
\end{align}
\end{enumerate}
\end{proposition}

\subsection{$L^2(D)$-valued Wiener and fractional Brownian processes}

Let $\set{\lambda_j, j\in\N}$  and $\set{\mu_j, j\in\N}$ be the sequences of positive real numbers and
$\set{e_i,i\in\N}$ be an orthonormal basis of $L^2(D)$.
Assume that
\begin{enumerate}[resume*]
\item\label{(A5)}
$\displaystyle\sup_j\norm{e_j}_\infty<\infty, \quad
\sum_{j=1}^\infty\lambda_j<\infty,
\quad\text{and}\quad
\sum_{j=1}^\infty\mu_j^{1/2}<\infty$.
\end{enumerate}

Let
$B_j=\set{B_j(t), t\ge0}$, $j\in\N$,
be a sequence of one-dimensional, independent Brownian motions defined on $(\Omega,\F,\FF,\Prob)$.
Define $L^2(D)$-valued Wiener process $W=\set{W(\cdot,t),t\ge0}$ by
\[
W(\cdot,t) = \sum_{j=1}^\infty \lambda_j^{1/2} e_j(\cdot) B_j(t),
\]
where the series converges in $L^2(\Omega,\F,\Prob)$, see, e.\,g., \cite[Sec.~4.1]{DaPrato14} or \cite[Sec.~3.5]{Duan14}.

Similarly, let
$B^H_j=\set{B^H_j(t), t\ge0}$, $j\in\N$,
be a sequence of one-dimensional, independent fractional Brownian motions with the Hurst parameter $H\in(0,1)$, defined on $(\Omega,\F,\FF,\Prob)$ and starting at the origin.
Following \cite{Maslowski03}, define $L^2(D)$-valued fractional Brownian process $W^H=\set{W^H(\cdot,t),t\ge0}$ by
\begin{equation}\label{eq:WH}
W^H(\cdot,t) = \sum_{j=1}^\infty \mu_j^{1/2} e_j(\cdot) B^H_j(t),
\end{equation}
where the series converges a.\,s.\ in $L^2(D)$.

\begin{remark}
The assumption
$\sum_{j=1}^\infty\mu_j<\infty$
is sufficient for the series \eqref{eq:WH} to converge in $L_2(D)$.
The stronger condition
$\sum_{j=1}^\infty\mu_j^{1/2}<\infty$
is needed for the definition of a stochastic integral with respect to this process, see Subsection~\ref{sss:int-WH} below.
\end{remark}

\subsection{Stochastic integration with respect to $L^2(D)$-valued processes}
Let $\Phi = \set{\Phi(t), t\in[0,T]}$ be an adapted stochastic process taking values in the space of linear operators on $L^2(D)$, $\psi = \set{\psi(t), t\in[0,T]}$ be an adapted process in $L^\infty(D)$.
\subsubsection{Integration with respect to $W$}
Define the stochastic integral with respect to $L_2(D)$-valued Wiener process $W$ by
\[
\int_0^t \Phi(s)\psi(s)\,dW(s) \coloneqq \sum_{j=1}^\infty \lambda_j^{1/2} \int_0^t \Phi(s) (\psi(s) e_j) \,dB_j(s),
\]
where the integrals with respect to $B_j$, $j\in\N$, are It\^o integrals, see \cite{DaPrato14,Duan14,Zabczyk01}.
The It\^o isometry of the form
\[
\E\norm{\int_0^t \Phi(s) \psi(s)\,dW(s)}_2^2 = \E \sum_{j=1}^\infty \lambda_j \int_0^t \norm{\Phi(s) (\psi(s)e_j)}_2^2 \,ds,
\]
holds if the right-hand side is finite \cite[Sec.~4]{Zabczyk01}.
Moreover, one has the following version of the Burkholder--Davis--Gundy inequality from \cite[Th.~4.36]{DaPrato14} (see also \cite[Lemma 3.24]{Duan14}):
for all $p\ge2$ there exists $C_p>0$ such that for all $t\in[0,T]$,
\[
\E\sup_{0\le s\le t}\norm{\int_0^s \Phi(v) \psi(v)\,dW(v)}_2^{p}
 \le C_p\E\left[\sum_{j=1}^\infty \lambda_j \int_0^t \norm{\Phi(s) (\psi(s)e_j)}_2^2 \,ds\right]^{p/2}.
\]
This result can be generalized to the case of stochastic convolutions (see \cite{Kotelenez84,Tubaro84}). For simplicity we formulate it for the evolution family $U(t,s)$ defined by \eqref{eq:U(t,s)}. Define
\begin{align*}
S(\cdot,t) = \int_0^t U(t,s)\psi(s)\,dW(s)
= \sum_{j=1}^\infty \lambda_j^{1/2} \int_0^t \int_D G(\cdot,t,y,s) \psi(y,s) e_j(y)dy \,dB_j(s).
\end{align*}
Then for every $p\ge2$ there exists $C_p'>0$ such that for all $t\in[0,T]$,
\begin{equation}\label{eq:Burk}
\E\sup_{0\le s\le t}\norm{S(\cdot,s)}_2^{p}
 \le C_p'\E\left[\int_0^t \norm{\psi(\cdot,s)}_2^2 \,ds\right]^{p/2}.
\end{equation}

\subsubsection{Integration with respect to $W^H$}
\label{sss:int-WH}
In order to introduce the integral with respect to $L_2(D)$-valued fractional Brownian process $W^H$ with $H>1/2$, we need to recall the definition of generalized Lebesgue--Stieltjes integral.
Let $f,g\colon[a,b]\to \R$ be two functions and $\alpha\in(0,1)$.
Then the Riemann--Liouville left- and right-sided fractional derivatives are defined by
\begin{align*}
\Der^\alpha_{a+} f(x) &= \frac1{\Gamma(1-\alpha)}\left(\frac{f(x)}{(x-a)^\alpha}+
\alpha\int_a^x\frac{f(x)-f(y)}{(x-y)^{\alpha+1}}\,dy\right),
\\
\Der^{1-\alpha}_{b-} g(x) &= \frac{1}{\Gamma(\alpha)}\left(\frac{g(x)}{(b-x)^{1-\alpha}}+
(1-\alpha)\int_x^b\frac{g(x)-g(y)}{(y-x)^{2-\alpha}}\,dy\right).
\end{align*}
Assume that
$\Der^\alpha_{a+} f\in L_1[a,b]$,
$\Der^{1-\alpha}_{b-} g_{b-}\in L_\infty[a,b]$,
where
$g_{b-}(x)=g(b-)-g(x)$.
Under these assumptions, the generalized Lebesgue--Stieltjes integral $\int_a^bf(x)dg(x)$ is defined by
\begin{equation}\label{eq:gen-int}
\int_a^bf(x)dg(x)=\int_a^b
\Der^\alpha_{a+}f(x)\,\Der^{1-\alpha}_{b-}g_{b-}(x)\,dx,
\end{equation}
see, e.\,g., \cite[Sec. 2.1]{mishura} for details.
It is not hard to see that this integral admits the bound
\begin{equation}\label{eq:bound-int}
\abs{\int_a^bf(x)dg(x)}\le C_\alpha \norm{g}_{\alpha,0;b}\int_a^b
\left(\frac{\abs{f(x)}}{(x-a)^\alpha}+
\int_a^x\frac{\abs{f(x)-f(y)}}{(x-y)^{\alpha+1}}\,dy\right)dx,
\end{equation}

Fix $H\in(1/2,1)$, $\alpha\in(1-H,1/2)$, and let $\Phi,\psi$ be as above. Assume that
\[
\sup_{j\in\N}\norm{\Phi(s)(\psi(s) e_j)}_{\alpha,1,b}<\infty,
\quad b,s\in\R^+.
\]
Following  \cite{Maslowski03}, we introduce the integral with respect to $L^2(D)$-valued fractional Brownian process by
\[
\int_a^b \Phi(s)\psi(s)\,dB^H(s) \coloneqq \sum_{j=1}^\infty \mu_j^{1/2} \int_a^b \Phi(s)(\psi(s) e_j) \,dB^H_j(s),
\]
where the integrals with respect to $B^H_j$, $j\in\N$, are pathwise generalized Lebesgue--Stieltjes integrals defined by \eqref{eq:gen-int}.
One can easily derive  from \eqref{eq:bound-int} the following inequality (see \cite[Eq. (2.16)]{Maslowski03})
\begin{equation}\label{eq:bound-int-2}
\begin{aligned}
\norm{\int_a^b \Phi(s)\psi(s)\,dB^H(s)}_2
&\le C_\alpha \xi_{\alpha,H,b} \sup_{j\in\N}
\int_a^b\left(\frac{\norm{\Phi(s)(\psi(s)e_j)}_2}{(s-a)^\alpha}\right. \\
&\left.
+  \int_a^s \frac{\norm{(\Phi(s)(\psi(s)e_j) - \Phi(v)(\psi(v)e_j)}_2}{(s-v)^{\alpha+1}}\,dv \right ) ds,
\end{aligned}
\end{equation}
where
\begin{equation}\label{eq:xi}
\xi_{\alpha,H,b}\coloneqq\sum_{j=1}^\infty \mu_j^{1/2} \norm{B^H_j}_{\alpha,0;b}.
\end{equation}
Note that the value
$\E\norm{B^H_j}_{\alpha,0;b}$ is finite by \cite[Lemma~7.5]{NR02}, moreover, it does not depend on $j$, since $B^H_j$'s are equally distributed.
Hence, using the monotone convergence theorem and assumption \ref{(A5)}, we get
\[
\E\sum_{j=1}^\infty \mu_j^{1/2} \norm{B^H_j}_{\alpha,0;b}
=\sum_{j=1}^\infty \mu_j^{1/2} \E\norm{B^H_j}_{\alpha,0;b}<\infty.
\]
Therefore the random variable
$\xi_{\alpha,H,b}$
is finite a.\,s.

\section{Existence and uniqueness of mild solution}
\label{sec:mild}

In this section, we consider unique solvability of the problem \eqref{eq:spde}--\eqref{eq:boundary}.
We understand its solution in a mild sense.
Recall that we consider $H\in(1/2,1)$.

\begin{definition}
$L^2(D)$-valued random field $\set{u(\cdot,t),t\in[0,T]}$ is a \emph{mild solution} to the problem \eqref{eq:spde}--\eqref{eq:boundary} if the following two conditions are satisfied:
\begin{enumerate}[(1)]
\item $u\in\B^{\alpha,2}\left (0,T;L^2(D)\right)$ a.\,s.\ for some $\alpha\in(1-H,1/2)$,
\item
the relation \eqref{eq:mild1}, equivalently,
\begin{align}
u(\cdot,t) &= \int_D G(\cdot,t;y,0)\varphi(y)\,dy
+ \int_0^t\int_D G(\cdot,t;y,s)f(u(y,s))\,dy\,ds \notag
\\
&\quad+ \sum_{j=1}^\infty \lambda_j^{1/2} \int_0^t \int_D G(\cdot,t;y,s)g(u(y,s))e_j(y)\,dy\,dB_j(s) \notag
\\
&\quad+ \sum_{j=1}^\infty \mu_j^{1/2} \int_0^t \int_D G(\cdot,t;y,s)h(u(y,s))e_j(y)\,dy\,dB^H_j(s)
\label{eq:mild}
\end{align}
holds a.\,s.\ for every $t\in[0,T]$ as an equality in $L^2(D)$.
\end{enumerate}
Here the integrals w.\,r.\,t.\ $B_j$, $j\in\N$, are It\^o integrals, and the integrals w.\,r.\,t.\ $B^H_j$, $j\in\N$, are path-wise generalized Lebesgue--Stieltjes integrals.
\end{definition}
For clarity, in the following we will write the integrals with respect to $W$ and $W^H$ in their full form, as in \eqref{eq:mild}.

\begin{theorem}\label{th:mild}
Assume that Hypotheses \ref{(A1)}--\ref{(A5)} hold, $H\in\left(\frac{d+1}{d+2},1\right)$.
Then the problem \eqref{eq:spde}--\eqref{eq:boundary} has a unique mild solution.
\end{theorem}

The proof will be divided into several logical steps.
First fix some $\alpha\in\left(1-H,\frac{1}{d+2}\right)$.

\subsection{Construction of approximations}

Fix $N\ge1$.
Let
\[
\tau_N = \inf\set{t:\xi_{\alpha,H,t}\ge N}\wedge T,
\]
where $\xi_{\alpha,H,t}$ is defined in \eqref{eq:xi}.
Put $B^{H,N}_j(t)=B^{H}_j(t\wedge\tau_N)$, $t\in[0,T]$, $j\in\N$.
For each $n,j\in\N$ define a smooth approximation of $B^{H,N}_j$ by
\begin{equation}\label{eq:approx-B}
B^{H,N,n}_j(t) = n \int_{(t-1/n)\vee0}^t B^{H,N}_j(s)\,ds,
\end{equation}
see \cite{MSh12}.
Consider the equation
\begin{align}
u_{N,n}(\cdot,t) &= \int_D G(\cdot,t;y,0)\varphi(y)\,dy
+ \int_0^t\int_D G(\cdot,t;y,s)f(u_{N,n}(y,s))\,dy\,ds
\notag
\\
&\quad+ \sum_{j=1}^\infty \lambda_j^{1/2} \int_0^t \int_D G(\cdot,t;y,s)g(u_{N,n}(y,s))e_j(y)\,dy\,dB_j(s)
\notag
\\
&\quad+ \sum_{j=1}^\infty \mu_j^{1/2} \int_0^t \int_D G(\cdot,t;y,s)h(u_{N,n}(y,s))e_j(y)\,dy\,\frac{d}{ds}B^{H,N,n}_j(s)\,ds,
\label{eq:approx-u}
\end{align}
or
\begin{align*}
u_{N,n}(\cdot,t) &= \int_D G(\cdot,t;y,0)\varphi(y)\,dy
+ \int_0^t\int_D G(\cdot,t;y,s)b^{N,n}(u_{N,n}(y,s),y,\omega,s)\,dy\,ds
\\
&\quad+ \sum_{j=1}^\infty \lambda_j^{1/2} \int_0^t \int_D G(\cdot,t;y,s)g(u_{N,n}(y,s))e_j(y)\,dy\,dB_j(s),
\end{align*}
where
\[
b^{N,n}(u,\omega,x,s)=f(u)+h(u)\sum_{j=1}^\infty \mu_j^{1/2} e_j(x)\frac{d}{ds}B^{H,N,n}_j(s)
\]
is a random drift depending also on $(y,s)$.
In other words, $u_{N,n}$ is a mild solution of the equation
\[
du(x,t) = \Bigl(\Div\bigl(k(x,t)\nabla u(x,t)\bigr)+b^{N,n}\bigl(u(x,t),x,\omega,t\bigr)\Bigr)dt
+ g\bigl(u(x,t)\bigr) W(x,dt),
\]
$(x,t) \in D\times[0,T]$,
with initial-boundary conditions \eqref{eq:init}--\eqref{eq:boundary}.
Such equations were studied in \cite[Sec.~8]{Veraar10}.

By assumption \ref{(A5)},
\begin{align*}
\MoveEqLeft
\abs{\sum_{j=1}^\infty \mu_j^{1/2} e_j(y)\frac{d}{ds}B^{H,N,n}_j(s)}
\le
C\sum_{j=1}^\infty \mu_j^{1/2} \abs{\frac{d}{ds}B^{H,N,n}_j(s)}
\\
&= Cn\sum_{j=1}^\infty \mu_j^{1/2} \abs{B^{H,N}_j(s)-B^{H,N}_j\left((s-\tfrac1n)\vee0\right)}
\\
&\le Cn^{\alpha}\sum_{j=1}^\infty \mu_j^{1/2} \norm{B^{H,N}_j}_{0,\alpha,s}
\le C n^{\alpha} N.
\end{align*}
Therefore, the function $b^{N,n}$ satisfies the following conditions:
for all $y\in D$, $\omega\in\Omega$, $s\in[0,T]$ and $u,v\in\R$,
\begin{gather}
\abs{b^{N,n}(u,y,\omega,s)-b^{N,n}(v,y,\omega,s)}\le C\abs{u-v},
\\
\abs{b^{N,n}(u,y,\omega,s)}\le C(1+\abs{u}).
\end{gather}
Then, by \cite[Example~8.2]{Veraar10}, there exists a unique mild solution $u_{N,n}$
with paths in $C\left ([0,T];L^2(D)\right )$ a.\,s.
Moreover, $u_{N,n}$ belongs to $C^{\beta_1,\beta_2}\left (\oD\times[0,T]\right)$
for all $\beta_1\in(0,1)$, $\beta_2\in(0,1/2)$.

\subsection{Convergence of approximations}
Let us prove that, for a fixed $N\ge1$, the sequence
$\set{u_{N,n}, n\ge1}$
is fundamental in probability in the norm
$\norm{\,\cdot\,}_{\alpha,2,T}$.
For all $\eps>0$, $R\ge1$ and $n,m\in\N$, we have
\begin{align}
\MoveEqLeft
\Prob\left(\norm{u_{N,n}-u_{N,m}}_{\alpha,2,T}>\eps\right)
\notag
\\
&\le\Prob\left(\norm{u_{N,n}-u_{N,m}}_{\alpha,2,T}>\eps, \norm{u_{N,n}}_{\alpha,2,T}\le R, \norm{u_{N,m}}_{\alpha,2,T}\le R\right)
\notag
\\
&\quad+ \Prob\left(\norm{u_{N,n}}_{\alpha,2,T}>R\right)
+ \Prob\left(\norm{u_{N,m}}_{\alpha,2,T}>R\right).
\label{eq:prob}
\end{align}
By \cite[Prop.~2.1]{MSh12}, for any $j\in\N$,
$\norm{B^{H,N,n}_j-B^{H,N}_j}_{\alpha,0;T} \to 0$, $n\to\infty$, a.\,s.
Then, in view of the boundedness,
\begin{equation}\label{eq:conv-fbm}
\E\norm{B^{H,N,n}_j-B^{H,N}_j}_{\alpha,0;T}^2 \to 0, \quad n\to\infty.
\end{equation}
Since $B^{H,N,n}_j-B^{H,N}_j$, $j\in\N$, are identically distributed, we see that
\[
\E\sum_{j=1}^\infty\mu_j^{1/2}\norm{B^{H,N,n}_j-B^{H,N}_j}_{\alpha,0;T}^2 \to 0, \quad n\to\infty,
\]
by \ref{(A5)}.
Then, by the Cauchy--Schwarz inequality,
\begin{align}
\MoveEqLeft
\E\left(\sum_{j=1}^\infty\mu_j^{1/2}\norm{B^{H,N,n}_j-B^{H,N}_j}_{\alpha,0;T}\right)^2
\notag
\\
&\le \E\left[\sum_{j=1}^\infty\mu_j^{1/2}\norm{B^{H,N,n}_j-B^{H,N}_j}_{\alpha,0;T}^2\right] \left(\sum_{j=1}^\infty\mu_j^{1/2}\right)\to 0, \quad n\to\infty.
\label{eq:conv-fbm2}
\end{align}
Therefore, using Lemma~\ref{l:3} and Markov's inequality, we see that for all $\eps>0$, $R\ge1$,
\[
\Prob\left(\norm{u_{N,n}-u_{N,m}}_{\alpha,2,T}>\eps, \norm{u_{N,n}}_{\alpha,2,T}\le R, \norm{u_{N,m}}_{\alpha,2,T}\le R\right)
\to0,\quad n\to\infty,
\]
and
\[
\limsup_{n,m\to\infty}\Prob\left(\norm{u_{N,n}-u_{N,m}}_{\alpha,2,T}>\eps\right)
\le2\sup_{n\in\N} \Prob\left(\norm{u_{N,n}}_{\alpha,2,T}>R\right),
\]
by \eqref{eq:prob}.

Moreover, it follows from the convergence~\eqref{eq:conv-fbm} that
$\sup_{n\in\N}\E\norm{B^{H,N,n}_j}_{\alpha,0;T}^2<\infty$.
Hence, Lemma~\ref{l:2} and Markov's inequality imply that
\[
\sup_{n\in\N} \Prob\left(\norm{u_{N,n}}_{\alpha,2,T}>R\right)\to0,
\quad R\to\infty.
\]
Therefore, $\norm{u_{N,n}-u_{N,m}}_{\alpha,2,T}\to0$, $n,m\to\infty$, in probability.
Consequently, there exists a random process $u_{N}$ such that
$\norm{u_{N,n}-u_{N}}_{\alpha,2,T}\to0$, $n\to\infty$, in probability.
Then there exists an a.\,s.\ convergent subsequence, and without loss of generality we can assume that
\begin{equation}\label{eq:conv-sol}
\norm{u_{N,n}-u_{N}}_{\alpha,2,T}\to0,\quad n\to\infty, \text{ a.\,s.}
\end{equation}

\subsection{The limit provides a solution}
We have
\begin{align*}
\norm{u_{N,n}(\cdot,t)-u_{N}(\cdot,t)}_2
&\le C\left(\norm{\tilde I_{\Delta f}(\cdot,t)}_2^2 + \norm{\tilde I_{\Delta g}(\cdot,t)}_2^2 + \norm{\tilde I_{\Delta h}(\cdot,t)}_2^2
+\norm{\tilde I_{\Delta Z}(\cdot,t)}_2^2\right),
\end{align*}
where
\begin{align*}
\tilde I_{\Delta f}(x,t) &\coloneqq \int_0^t\int_D G(x,t;y,s)\bigl(f(u_{N,n}(y,s))-f(u_{N}(y,s))\bigr)\,dy\,ds,
\\
\tilde I_{\Delta g}(x,t) &\coloneqq \sum_{j=1}^\infty \lambda_j^{1/2} \int_0^t \int_D G(x,t;y,s) \bigl(g(u_{N,n}(y,s))-g(u_{N}(y,s))\bigr)e_j(y)\,dy\,dB_j(s),
\\
\tilde I_{\Delta h}(x,t) &\coloneqq \sum_{j=1}^\infty \mu_j^{1/2} \int_0^t \int_D G(x,t;y,s)\bigl(h(u_{N,n}(y,s))-h(u_{N}(y,s))\bigr)e_j(y)\,dy\,dB^{H,N,n}_j(s),
\\
\tilde I_{\Delta Z}(x,t) &\coloneqq \sum_{j=1}^\infty \mu_j^{1/2} \int_0^t \int_D G(x,t;y,s)h(u_{N}(y,s))e_j(y)\,dy\,\bigl(B^{H,N,n}_j-B^{H,N}_j\bigr)(ds).
\end{align*}
Consequently, in order to prove that $u_{N}$ satisfies \eqref{eq:mild}
with $B^H_j$ replaced by $B^{H,N}_j$, we need to show that these four integrals converge to zero.
Note that they can be bounded exactly in the same way as the corresponding integrals in Lemma~\ref{l:3} of Appendix \ref{sec:aux}. Denoting
\[
\ind_T=\ind_{\set{\norm{u_{N,n}}_{\alpha,2,T}\le R, \norm{u_{N}}_{\alpha,2,T}\le R}},
\]
we will obtain
\begin{align*}
\norm{\tilde I_{\Delta f}(\cdot,T)}^2 &\le C \int_0^T \norm{u_{N,n}-u_{N}}_{\alpha,2,s}^2 ds,
\\
\norm{\tilde I_{\Delta h}(\cdot,T)}^2 &\le C_N\int_0^T\norm{u_{N,n}-u_{N}}_{\alpha,2,s}^2ds,
\\
\norm{\tilde I_{\Delta Z}(\cdot,T)}\ind_T &\le
C R \sum_{j=1}^\infty\mu_j^{1/2}\norm{B^{H,N,n}_j-B^{H,N}_j}_{\alpha,0;T},
\end{align*}
and
\[
\E\left[\norm{\tilde I_{\Delta g}(\cdot,T)}_2^2 \ind_T\right]
\le C  \int_0^T\E\left[\norm{u_{N,n}-u_{N}}_{\alpha,2,s}^2\ind_s\right] ds.
\]
Taking into account \eqref{eq:conv-fbm2} and \eqref{eq:conv-sol}, we get that
$\E\norm{u_{N,n}(\cdot,T)-u_{N}(\cdot,T)}_2^2\ind_T\to0$ as $n\to\infty$, and, consequently, $\norm{u_{N,n}(\cdot,T)-u_{N}(\cdot,T)}_2\ind_T\to0$  in probability.
Thanks to the convergence
$\norm{u_{N,n}-u_{N}}_{\alpha,2,T}\to0$, $n\to\infty$, the event
$\set{\norm{u_{N}}_{\alpha,2,T}\le R}$
implies
$\set{\norm{u_{N,n}}_{\alpha,2,T}\le R}$
for $n$ large enough, therefore we have the convergence of the integrals in probability on
$\set{\norm{u_{N}}_{\alpha,2,T}\le R}$
and arbitrary $R \ge 1$, therefore
on $\Omega$.

\subsection{Letting $N\to\infty$ and uniqueness}
From Lemma \ref{l:3} and Remark \ref{rem:l3} after it, it is obvious that the processes $u_{N}$ and $u_{M}$ with $M \ge N$ coincide
a.\,s.\ on the set
$A_{N,T}=\set{\xi_{\alpha,H,T}\le N}$.
Therefore, there exists a process $u$ such that
for each $N \ge 1$, $u_N = u$ a.\,s.\ on $A_{N,T}$.
Consequently, $u$ satisfies \eqref{eq:mild} on each of the sets
$A_{N,T}$, $N\ge1$, hence, almost surely.

Finally, the uniqueness also follows from Lemma \ref{l:3}: each solution to \eqref{eq:mild} must coincide
with $u$ on each of the sets $A_{N,T}$, hence, almost surely.

\appendix
\section{Auxiliary results}
\label{sec:aux}

Let $u$ be a mild solution, defined by \eqref{eq:mild}.
Introduce the following notation:
\begin{align*}
I_0(x,t) &\coloneqq \int_D G(x,t;y,0)\varphi(y)\,dy,
\\
I_f(x,t) &\coloneqq \int_0^t\int_D G(x,t;y,s)f(u(y,s))\,dy\,ds,
\\
I_g(x,t) &\coloneqq \sum_{j=1}^\infty \lambda_j^{1/2} \int_0^t \int_D G(x,t;y,s)g(u(y,s))e_j(y)\,dy\,dB_j(s),
\\
I_h(x,t) &\coloneqq \sum_{j=1}^\infty \mu_j^{1/2} \int_0^t \int_D G(x,t;y,s)h(u(y,s))e_j(y)\,dy\,dB^H_j(s).
\end{align*}
Also, let
\begin{equation}\label{eq:J-def}
J_a(t) \coloneqq \int_0^t \left( \int_0^s \frac{\norm{I_a(\cdot,s)-I_a(\cdot,v)}_2}{(s-v)^{\alpha+1}}\,dv \right )^2 ds,
\quad a\in\set{0,f,g,h}.
\end{equation}
Then
\[
\norm{I_a}_{\alpha,2,t}
= \sup_{s\in[0,t]}\norm{I_a(\cdot,s)}_{2}^2
+J_a(t),
\quad a\in\set{0,f,g,h}.
\]

\begin{lemma}
\label{l:1}
Let $N\ge1$. Define
\[
A_{N,t}:=\set{\xi_{\alpha,H,t}\le N}, \quad t\in[0,T],
\]
where $\xi_{\alpha,H,t}$ is given by \eqref{eq:xi}.
Then under assumptions of Theorem~\ref{th:mild},
\[
\norm{u}_{\alpha,2,t}^2 \le C_N\left(1+\int_0^t\norm{u}_{\alpha,2,s}^2ds
+ \sup_{s\in[0,t]}\norm{I_g(\cdot,s)}_2^2+J_g(t)\right)
\]
for all $\omega\in A_{N,t}$.
\end{lemma}

\begin{proof}
Fix $\omega\in A_{N,t}$.
It follows from \eqref{eq:mild} that
\begin{equation}\label{eq:u2}
\norm{u(\cdot,t)}_2^2 \le C\left(\norm{I_0(\cdot,t)}_2^2 + \norm{I_f(\cdot,t)}_2^2 + \norm{I_g(\cdot,t)}_2^2 + \norm{I_h(\cdot,t)}_2^2\right).
\end{equation}
Evidently,
\begin{equation}\label{eq:I02}
\norm{I_0(\cdot,s)}_{2}^2
= \int_D\left(\int_D G(x,s;y,0)\varphi(y)\,dy\right)^2dx
\le C,
\end{equation}
since $\varphi$ is bounded and the Gaussian property \eqref{eq:gaus} holds.
Using the Schwarz inequality and \eqref{eq:gaus}, we also get
\begin{align}
\norm{I_f(\cdot,t)}_2^2
&=\int_D\left(\int_0^t \int_D \abs{G(x,t;y,s)f(u(y,s))}dy\,ds\right)^2dx
\notag
\\
&\le C \int_D \left(\int_0^t \!\!\int_D \abs{G(x,t;y,s)}dy\,ds\right)\! \left(\int_0^t\!\! \int_D \abs{G(x,t;y,s)}\abs{f(u(y,s))}^2dy\,ds\right)dx
\notag
\\
&\le C \int_D \int_0^t \int_D \abs{G(x,t;y,s)}\abs{f(u(y,s))}^2 dy\,ds\,dx
\notag
\\
&\le C \int_0^t\int_D\left(\int_D \abs{G(x,t;y,s)}\,dx\right)\left(1+\abs{u(y,s)}^2\right) dy\,ds
\notag
\\
&\le C \int_0^t \int_D\left(1+\abs{u(y,s)}^2\right) dy\,ds
\le C \left(1 + \int_0^t \norm{u(\cdot,s)}_2^2ds \right).
\label{eq:If2}
\end{align}

Define
\begin{equation}\label{eq:f}
a_{j,t}(u)(x,s)\coloneqq \int_D G(x,t;y,s)h(u(y,s))e_j(y)\,dy.
\end{equation}
Then applying \eqref{eq:bound-int-2}, we obtain
\begin{align*}
\MoveEqLeft
\norm{I_h(\cdot,t)}_2 = \norm{\sum_{j=1}^\infty \mu_j^{1/2} \int_0^t a_{j,t}(u)(\cdot,s)\,dB^H_j(s)}_2
\\
&\le C \xi_{\alpha,H,t}\sup_{j\in\N} \int_0^t\left( \frac{\norm{a_{j,t}(u)(\cdot,s)}_2}{s^\alpha}
+\int_0^s \frac{\norm{a_{j,t}(u)(\cdot,s)-a_{j,t}(u)(\cdot,v)}_2}{(s-v)^{\alpha+1}}\,dv\right)ds.
\end{align*}
Taking into account that $\sup_{j\in\N}\norm{e_j}_\infty<\infty$, one can derive the following bound similarly to \eqref{eq:If2}:
\begin{equation}\label{eq:f-bound}
\sup_{j\in\N}\norm{a_{j,t}(u)(\cdot,s)}_2 \le C (1+\norm{u(\cdot,s)}_2).
\end{equation}
Further, by the assumption \ref{(A3)},
\begin{align*}
\MoveEqLeft
\abs{a_{j,t}(u)(x,s)-a_{j,t}(u)(x,v)}
\le C\int_D \abs{G(x,t;y,s)}\abs{u(y,s)-u(y,v)}\,dy
\\*
&\quad+C\int_D \abs{G(x,t;y,s)-G(x,t;y,v)}(1+\abs{u(y,v)})\,dy.
\end{align*}
Applying the Schwarz inequality, we obtain
\begin{align*}
\MoveEqLeft
\abs{a_{j,t}(u)(x,s)-a_{j,t}(u)(x,v)}^2
\le C\int_D \abs{G(x,t;y,s)}\abs{u(y,s)-u(y,v)}^2\,dy
\\*
&\quad+C\int_D \abs{G(x,t;y,s)-G(x,t;y,v)}(1+\abs{u(y,v)})^2\,dy,
\end{align*}
since the integrals $\int_D \abs{G(x,t;y,s)}dy$ and
$\int_D \abs{G(x,t;y,s)-G(x,t;y,v)}dy$ are bounded uniformly in $s,v$ due to the Gaussian property \eqref{eq:gaus}.
Then applying the bounds \eqref{eq:gaus} and \eqref{eq:deltaG1} and integrating the preceding estimate w.\,r.\,t.\ $x\in D$, we get
\begin{align*}
\sup_{j\in\N}\norm{a_{j,t}(u)(\cdot,s)-a_{j,t}(u)(\cdot,v)}^2_2
&\le C\Bigl(\norm{u(\cdot,s)-u(\cdot,v)}^2_2
\notag
\\*
&\quad+(t-s)^{-\delta} (s-v)^{\delta} \left(1 + \norm{u(\cdot,v)}^2_2\right)\Bigr),
\end{align*}
whence
\begin{align}
\sup_{j\in\N}\norm{a_{j,t}(u)(\cdot,s)-a_{j,t}(u)(\cdot,v)}_2
&\le C\Bigl(\norm{u(\cdot,s)-u(\cdot,v)}_2
\notag
\\*
&\quad+(t-s)^{-\frac\delta2} (s-v)^{\frac\delta2} \left(1 + \norm{u(\cdot,v)}_2\right)\Bigr).
\label{eq:df-bound}
\end{align}
for any $\delta\in\left(\frac{d}{d+2},1\right)$.
Therefore,
\begin{align}
\norm{I_h(\cdot,t)}_2 &\le C \xi_{\alpha,H,t} \left(1 + \int_0^t\frac{\norm{u(\cdot,s)}_2}{s^\alpha}\,ds
+ \int_0^t\int_0^s \frac{\norm{u(\cdot,s)-u(\cdot,v)}_2}{(s-v)^{\alpha+1}}\,dv\,ds
\right.\notag
\\
&\quad+\left. \int_0^t (t-s)^{-\delta/2} \int_0^s (s-v)^{\delta/2-\alpha-1} \left(1 + \norm{u(\cdot,v)}_2\right) \,dv\,ds \right),
\label{eq:Ih0}
\end{align}
where $\delta\in\left(\frac{d}{d+2},1\right)$ is arbitrary.
Since $\alpha<1/2$, we can choose $\delta>2\alpha$. Then by changing the order of integration, the last term can be calculated as follows
\begin{align}
\MoveEqLeft
\int_0^t (t-s)^{-\delta/2} \int_0^s (s-v)^{\delta/2-\alpha-1} \left(1 + \norm{u(\cdot,v)}_2\right) \,dv\,ds
\notag
\\*
&=\int_0^t \left(\int_v^t (t-s)^{-\delta/2}  (s-v)^{\delta/2-\alpha-1}ds\right) \left(1 + \norm{u(\cdot,v)}_2\right) \,dv
\notag
\\
&=\int_0^t (t-v)^{-\alpha}\left(\int_0^1 (1-z)^{-\delta/2}  z^{\delta/2-\alpha-1}dz\right) \left(1 + \norm{u(\cdot,v)}_2\right) \,dv
\notag
\\
&= \Beta\left(\tfrac\delta2-\alpha,1-\tfrac\delta2\right) \int_0^t (t-v)^{-\alpha} \left(1 + \norm{u(\cdot,v)}_2\right) \,dv.
\label{eq:calc-int}
\end{align}
We arrive at
\begin{align*}
\norm{I_h(\cdot,t)}_2 &\le C N \left(1 + \int_0^t\left(\frac{1}{s^\alpha}+\frac{1}{(t-s)^\alpha}\right)\norm{u(\cdot,s)}_2\,ds
\right.
\\
&\quad+\left. \int_0^t\int_0^s \frac{\norm{u(\cdot,s)-u(\cdot,v)}_2}{(s-v)^{\alpha+1}}\,dv\,ds\right).
\end{align*}
Applying the Schwarz inequality relative to the measure $ds$ on $[0,t]$ to both integrals, we deduce that
\begin{equation}\label{eq:Ih2}
\norm{I_h(\cdot,t)}_2^2 \le C N^2 \left(1 + \int_0^t\norm{u(\cdot,s)}_2^2\,ds
+ \int_0^t\left(\int_0^s \frac{\norm{u(\cdot,s)-u(\cdot,v)}_2}{(s-v)^{\alpha+1}}\,dv\right)^2ds\right).
\end{equation}

Combining \eqref{eq:u2}--\eqref{eq:If2} and \eqref{eq:Ih2}, we obtain
\begin{align*}
\norm{u(\cdot,t)}_2^2 &\le C N^2 \left(1 + \int_0^t\norm{u(\cdot,s)}_2^2\,ds
+ \int_0^t\left(\int_0^s \frac{\norm{u(\cdot,s)-u(\cdot,v)}_2}{(s-v)^{\alpha+1}}\,dv\right)^2ds\right)
\\
&\quad+C \norm{I_g(\cdot,t)}_2^2
\\
&\le C N^2 \left(1 + \int_0^t\sup_{v\in[0,s]}\norm{u(\cdot,v)}_2^2\,ds
+ \int_0^t\left(\int_0^s \frac{\norm{u(\cdot,s)-u(\cdot,v)}_2}{(s-v)^{\alpha+1}}\,dv\right)^2ds\right)
\\
&\quad+C \norm{I_g(\cdot,t)}_2^2.
\end{align*}
By definition of the norm $\norm{\,\cdot\,}_{\alpha,2,t}$,
\begin{align}
\norm{u(\cdot,t)}_{\alpha,2,t}^2
&\le C N^2 \left(1 + \int_0^t\sup_{v\in[0,s]}\norm{u(\cdot,v)}_2^2\,ds
+ \int_0^t\left(\int_0^s \frac{\norm{u(\cdot,s)-u(\cdot,v)}_2}{(s-v)^{\alpha+1}}\,dv\right)^2ds\right)
\notag
\\
&\quad+C \sup_{s\in[0,t]}\norm{I_g(\cdot,s)}_2^2.
\label{eq:bound1}
\end{align}

Obviously,
\begin{equation}\label{eq:term2}
\int_0^t\left(\int_0^s \frac{\norm{u(\cdot,s)-u(\cdot,v)}_2}{(s-v)^{\alpha+1}}\,dv\right)^2ds
\le C\bigl(J_0(t)+J_f(t)+J_g(t)+J_h(t)\bigr),
\end{equation}
where the terms in the right-hand side are defined in \eqref{eq:J-def}.

Using the boundedness of $\varphi$ and \eqref{eq:deltaG}, we get for any $\delta\in(\frac{d}{d+2},1)$ and some $v^*\in(v,s)$,
\begin{align*}
\norm{I_0(\cdot,s)-I_0(\cdot,v)}_2^2
&\le C \int_D\left(\int_D \abs{G(x,s;y,0)-G(x,v;y,0)}\,dy\right)^2dx
\\
&\le C \int_D\left( (s-v)^\delta v^{-\delta} \int_D (v^*)^{-d/2} \exp\set{-C\frac{\abs{x-y}^2}{v^*}}\,dy\right)^2dx
\\
&\le C  (s-v)^{2\delta} v^{-2\delta},
\end{align*}
whence
\begin{equation}\label{eq:J0}
J_0(t)\le C\int_0^t \left( \int_0^s (s-v)^{\delta-\alpha-1} v^{-\delta}\,dv \right )^2 ds \le C,
\end{equation}
since we can choose $\delta>\alpha$.


In order to estimate $J_f(t)$, we write
\begin{align*}
I_f(x,s)-I_f(x,v) &= \int_v^s\int_D G(x,s;y,z)f(u(y,z))\,dy\,dz
\\
&\quad+ \int_0^v\int_D \bigl(G(x,s;y,z)-G(x,v;y,z)\bigr) f(u(y,z))\,dy\,dz
\\
&\eqqcolon K'_f(x,s,v)+K''_f(x,s,v).
\end{align*}
Then,
\[
J_f(t) \le \int_0^t \left( \int_0^s \left(\norm{K'_f(\cdot,s,v)}_2+\norm{K''_f(\cdot,s,v)}_2\right)(s-v)^{-\alpha-1}\,dv \right )^2 ds.
\]
It is not hard to show that
\begin{align}
\norm{K'_f(\cdot,s,v)}_2 &\le C(s-v)^{1/2}\left(\int_v^s\left(1+\norm{u(\cdot,z)}_2^2\right)dz\right)^{1/2},
\label{eq:Kf'}
\\
\norm{K''_f(\cdot,s,v)}_2 &\le C(s-v)^{\delta/2}\left(\int_0^v (v-z)^{-\delta}\left(1+\norm{u(\cdot,z)}_2^2\right)dz\right)^{1/2},
\label{eq:Kf''}
\end{align}
for every $\delta\in\left(\frac{d}{d+2},1\right)$.
Indeed, \eqref{eq:Kf'} follows by applying the Schwarz inequality w.\,r.\,t.\ the finite measure $\abs{G(x,s;y,z)}dy\,dz$ on $D\times[v,s]$ and using the Gaussian property along with assumption \ref{(A2)}.
Inequality \eqref{eq:Kf''} is derived by applying
the Schwarz inequality w.\,r.\,t.\ the measure $\abs{G(x,s;y,z)-G(x,v;y,z)}dy\,dz$ on $D\times[0,v]$ and then \eqref{eq:deltaG}.
Further, \eqref{eq:Kf'} implies
\begin{align*}
\MoveEqLeft
\int_0^t \left( \int_0^s \norm{K'_f(\cdot,s,v)}_2 (s-v)^{-\alpha-1}\,dv \right )^2 ds
\\
&\le
C \int_0^t \left( \int_0^s (s-v)^{-1/2-\alpha} \left(\int_v^s\left(1+\norm{u(\cdot,z)}_2^2\right)dz\right)^{1/2}dv\right )^2 ds
\\
&\le
C \int_0^t \left(1+\sup_{z\in[0,s]}\norm{u(\cdot,z)}_2^2\right) \left( \int_0^s (s-v)^{-1/2-\alpha} dv\right )^2 ds
\\
&\le
C \int_0^t \left(1+\sup_{z\in[0,s]}\norm{u(\cdot,z)}_2^2\right)ds,
\end{align*}
since $\alpha<1/2$.
Applying the bound \eqref{eq:Kf''} and the Schwarz inequality, we obtain
\begin{align*}
\MoveEqLeft
\int_0^t \left( \int_0^s \norm{K''_f(\cdot,s,v)}_2 (s-v)^{-\alpha-1}\,dv \right )^2 ds
\\
&\le
C \int_0^t \left( \int_0^s (s-v)^{\delta/2-\alpha-1}\left(\int_0^v (v-z)^{-\delta}\left(1+\norm{u(\cdot,z)}_2^2\right)dz\right)^{1/2}dv \right )^2 ds
\\
&\le
C \int_0^t \left( \int_0^s (s-v)^{\delta/2-\alpha-1}dv\right)
\\*
&\quad\times\left( \int_0^s (s-v)^{\delta/2-\alpha-1}\int_0^v (v-z)^{-\delta}\left(1+\norm{u(\cdot,z)}_2^2\right)dz\,dv \right ) ds
\\
&\le
C \int_0^t \int_0^s (s-v)^{\delta/2-\alpha-1}\int_0^v (v-z)^{-\delta}\left(1+\norm{u(\cdot,z)}_2^2\right)dz\,dv\,ds
\\
&\le
C \int_0^t \left(1+\sup_{z\in[0,s]}\norm{u(\cdot,z)}_2^2\right)\int_0^s (s-v)^{\delta/2-\alpha-1}v^{1-\delta}\,dv\,ds
\\
&\le
C \int_0^t \left(1+\sup_{z\in[0,s]}\norm{u(\cdot,z)}_2^2\right)ds,
\end{align*}
where the last inequality follows from
\[
\int_0^s (s-v)^{\delta/2-\alpha-1}v^{1-\delta}\,dv
=\Beta(\delta/2-\alpha,2-\delta) s^{1-\delta/2-\alpha}
\le\Beta(\delta/2-\alpha,2-\delta) T^{1-\delta/2-\alpha},
\]
since $\delta/2+\alpha<1$.
Combining the above estimates, we arrive at
\begin{equation}\label{eq:Jf}
J_f(t)\le C\left(1+\int_0^t\norm{u}_{\alpha,2,s}^2ds\right).
\end{equation}

It remains to estimate $J_h$.
We start by writing
\begin{align*}
I_h(x,s)-I_h(x,v)
&=\sum_{j=1}^\infty \mu_j^{1/2} \int_v^s a_{j,s}(u)(x,z)\,dB^H_j(z)
\\
&\quad+ \sum_{j=1}^\infty \mu_j^{1/2} \int_0^v \left(a_{j,s}(u)(x,z)-a_{j,v}(u)(x,z)\right)dB^H_j(z).
\\
&= K'_h(x,s,v)+K''_h(x,s,v).
\end{align*}
where $a_{j,s}(u)(x,z)$ is defined by \eqref{eq:f}.
Then
\begin{equation}\label{eq:Jh1}
J_h(t) \le \int_0^t \left( \int_0^s \left(\norm{K'_h(\cdot,s,v)}_2+\norm{K''_h(\cdot,s,v)}_2\right)(s-v)^{-\alpha-1}\,dv \right )^2 ds.
\end{equation}
Applying \eqref{eq:bound-int-2}, \eqref{eq:f-bound} and \eqref{eq:df-bound}, we get
\begin{align*}
\norm{K'_h(\cdot,s,v)}_2
&\le C \xi_{\alpha,H,t}\sup_{j\in\N} \int_v^s\biggl( \frac{\norm{a_{j,s}(u)(\cdot,z)}_2}{(z-v)^\alpha}
\\
&\quad+\int_v^z \frac{\norm{a_{j,s}(u)(\cdot,z)-a_{j,s}(u)(\cdot,r)}_2}{(z-r)^{\alpha+1}}\,dr\biggr)dz
\\
&\le C \xi_{\alpha,H,t}\sup_{j\in\N} \biggl( \int_v^s\frac{1+\norm{u(\cdot,z)}_2}{(z-v)^\alpha}dz
+\int_v^s\!\!\int_v^z \frac{\norm{u(\cdot,z)-u(\cdot,r)}_2}{(z-r)^{\alpha+1}}\,dr\,dz
\\
&\quad+\int_v^s\!\!\int_v^z (s-z)^{-\delta/2} (z-r)^{\delta/2-\alpha-1} \left(1+\norm{u(\cdot,r)}_2\right)dr\,dz
\biggr).
\end{align*}
for any $\delta\in\left(\frac{d}{d+2},1\right)$.
The last term is computed similarly to \eqref{eq:calc-int} (recall that $\delta>2\alpha$):
\begin{align*}
\MoveEqLeft
\int_v^s\!\!\int_v^z (s-z)^{-\delta/2} (z-r)^{\delta/2-\alpha-1} \left(1+\norm{u(\cdot,r)}_2\right)dr\,dz
\\*
&= \Beta\left(\tfrac\delta2-\alpha,1-\tfrac\delta2\right) \int_s^v (s-r)^{-\alpha} \left(1 + \norm{u(\cdot,r)}_2\right) \,dr.
\end{align*}
Thus,
\begin{align}
\norm{K'_h(\cdot,s,v) }_2 &\le C \xi_{\alpha,H,t}
\left(\int_v^s\left(\frac{1}{(z-v)^\alpha}+\frac{1}{(s-z)^\alpha}\right)(1+\norm{u(\cdot,z)}_2)\,dz
\right.
\notag
\\*
&\quad+\left. \int_v^s\int_v^z \frac{\norm{u(\cdot,z)-u(\cdot,r)}_2}{(z-r)^{\alpha+1}}\,dr\,dz\right).
\label{eq:Kh'}
\end{align}

Denoting
$a^*_{j,s,v}(u)(x,z)=a_{j,s}(u)(x,z)-a_{j,v}(u)(x,z)$,
we can write by \eqref{eq:bound-int-2},
\begin{align*}
\norm{K''_h(\cdot,s,v)}_2
&\le C \xi_{\alpha,H,t}\sup_{j\in\N} \int_0^v\biggl( \frac{\norm{a^*_{j,s,v}(u)(\cdot,z)}_2}{z^\alpha}
\\*
&+\int_0^z \frac{\norm{a^*_{j,s,v}(u)(\cdot,z)-a^*_{j,s,v}(u)(\cdot,r)}_2}{(z-r)^{\alpha+1}}\,dr\biggr)dz.
\end{align*}

Similarly to \eqref{eq:f-bound} and \eqref{eq:df-bound}, we can prove the inequalities
\begin{equation}\label{eq:f*-bound}
\norm{a^*_{j,s,v}(u)(\cdot,z)}_2 \le C (s-v)^{\frac\delta2} (v-z)^{-\frac\delta2} (1+\norm{u(\cdot,z)}_2),
\end{equation}
and
\begin{align}
\norm{a^*_{j,s,v}(u)(\cdot,z)-a^*_{j,s,v}(u)(\cdot,r)}_2
&\le C(s-v)^{\frac\delta2} \Bigl( (v-z)^{-\frac\delta2} \norm{u(\cdot,z)-u(\cdot,r)}_2
\notag
\\
&\quad+(v-z)^{-\frac12}(z-r)^{-\frac12(1-\delta)} \left(1 + \norm{u(\cdot,r)}_2\right)\Bigr).
\label{eq:df*-bound}
\end{align}
for any $\delta\in\left(\frac{d}{d+2},1\right)$.
For \eqref{eq:f*-bound} the key estimate is \eqref{eq:deltaG}.
For \eqref{eq:df*-bound} one should apply \eqref{eq:deltaG} along with \eqref{eq:deltaG2}.

Then
\begin{align}
\norm{K''_h(\cdot,s,v) }_2 &\le C \xi_{\alpha,H,t}(s-v)^{\delta/2}
\left(\int_0^v(v-z)^{-\delta/2}\left(\frac{1}{z^\alpha}+\frac{1}{(v-z)^\alpha}\right)(1+\norm{u(\cdot,z)}_2)\,dz
\right.
\notag
\\*
&\quad+\left. \int_0^v(v-z)^{-\delta/2}\int_v^z \frac{\norm{u(\cdot,z)-u(\cdot,r)}_2}{(z-r)^{\alpha+1}}\,dr\,dz\right).
\label{eq:Kh''}
\end{align}

Thus, combining \eqref{eq:Jh1}, \eqref{eq:Kh'} and \eqref{eq:Kh''}, we get
\[
J_h(t)\le C N^2 \int_0^t \left[(L_1(s))^2 +(L_2(s))^2+(L_3(s))^2+(L_4(s))^2\right]ds,
\]
where
\begin{align*}
L_1(s) &= \int_0^s (s-v)^{-\alpha-1} \int_v^s \left(\frac{1}{(z-v)^\alpha} + \frac{1}{(s-z)^\alpha}\right) (1+\norm{u(\cdot,z)}_2)\,dz\,dv,
\\
L_2(s) &= \int_0^s (s-v)^{-\alpha-1} \int_v^s\int_v^z \frac{\norm{u(\cdot,z)-u(\cdot,r)}_2}{(z-r)^{\alpha+1}}\,dr\,dz\,dv,
\\
L_3(s) &= \int_0^s (s-v)^{\delta/2-\alpha-1}
\int_0^v (v-z)^{-\delta/2} \left(\frac{1}{z^\alpha} + \frac{1}{(v-z)^\alpha}\right) (1+\norm{u(\cdot,z)}_2)\,dz\,dv,
\\
L_4(s) &= \int_0^s (s-v)^{\delta/2-\alpha-1}
\int_0^v(v-z)^{-\delta/2}\int_v^z \frac{\norm{u(\cdot,z)-u(\cdot,r)}_2}{(z-r)^{\alpha+1}}\,dr\,dz\,dv.
\end{align*}
Similarly to Eqs.~(3.42)--(3.45) of \cite{SSV09}, one can estimate the integrals
$\int_0^t(L_i(s))^2\,ds$, $i=1,2,3,4$.
This leads to the inequality
\begin{equation}\label{eq:Jh}
J_h(t)\le CN^2\left(1+\int_0^t\norm{u}_{\alpha,2,s}^2ds\right).
\end{equation}
Finally, combining \eqref{eq:term2}, \eqref{eq:J0}, \eqref{eq:Jf} and \eqref{eq:Jh}, we get
\[
\int_0^t\left(\int_0^s \frac{\norm{u(\cdot,s)-u(\cdot,v)}_2}{(s-v)^{\alpha+1}}\,dv\right)^2ds
\le CN^2\left(1+\int_0^t\norm{u}_{\alpha,2,s}^2ds\right)+ CJ_g(t).
\]
Inserting this bound into \eqref{eq:bound1}, we conclude the proof.
\end{proof}

\begin{lemma}\label{l:2}
Let $N\ge1$. Define a stopping time
\[
\tau_N=\inf\set{t:\xi_{\alpha,H,t}\ge N}\wedge T
\]
and a stopped process
$u^N(\cdot,t)=u(\cdot,t\wedge\tau_N)$.
Then
\[
\E\left[\norm{u^N}_{\alpha,2,T}^2\right] \le C_N.
\]
\end{lemma}
\begin{proof}
By Lemma~\ref{l:1},
\begin{equation}\label{eq:l2a}
\E\left[\norm{u^N}_{\alpha,2,t}^2\right]
\le C_N\left(1+\int_0^t\E\left[\norm{u^N}_{\alpha,2,s}^2\right]ds
+ \E \left[\sup_{s\in[0,t]}\norm{I_g^N(\cdot,s)}_2^2\right]
+ \E \left[J_g^N(t)\right]\right),
\end{equation}
where
\begin{gather*}
I_g^N(x,t) = \sum_{j=1}^\infty \lambda_j^{1/2} \int_0^t \int_D G(x,t;y,s)g\left(u^N(y,s)\right)e_j(y)\,dy\,dB_j(s),
\\
J_g^N(t) = \int_0^t \left( \int_0^s \frac{\norm{I_g^N(\cdot,s)-I_g^N(\cdot,v)}_2}{(s-v)^{\alpha+1}}\,dv \right )^2 ds.
\end{gather*}

Using the inequality \eqref{eq:Burk} and then the assumption \ref{(A2)}, we get
\begin{align}
\E \sup_{s\in[0,t]}\norm{I_g^N(\cdot,s)}_2^2
&\le C \E \int_0^t \norm{g\left(u^N(\cdot,s)\right)}_2^2\,ds
\notag
\le C \E \int_0^t\left(1+\norm{u^N(\cdot,s)}_2^2\right)\,ds
\notag
\\
&\le C \left(1+\E \int_0^t \sup_{z\in[0,s]}\norm{u^N(\cdot,z)}_2^2\,ds \right)
\notag
\\
&\le C \left(1+ \int_0^t\E\left[\norm{u^N}_{\alpha,2,s}^2\right] ds\right).
\label{eq:l2b}
\end{align}

By the Schwarz inequality, we have
\begin{align*}
\E J_g^N(t) &\le \E \int_0^t \left( \int_0^s \norm{I_g^N(\cdot,s)-I_g^N(\cdot,v)}_2^2 (s-v)^{-3/2-\alpha}\,dv \right )\left( \int_0^s(s-v)^{-1/2-\alpha}\,dv \right ) ds
\\
&\le C \int_0^t  \int_0^s \E \norm{I_g^N(\cdot,s)-I_g^N(\cdot,v)}_2^2 (s-v)^{-3/2-\alpha}\,dv\, ds
\\
&\le C \int_0^t  \int_0^s
\left(\E\norm{K'_g(\cdot,s,v)}_2^2+\E \norm{K''_g(\cdot,s,v)}_2^2\right) (s-v)^{-3/2-\alpha}\,dv \,ds,
\end{align*}
where
\begin{align*}
K'_g(x,s,v) &= \sum_{j=1}^\infty \lambda_j^{1/2} \int_v^s \int_D G(x,s;y,z)g\left(u^N(y,z)\right)e_j(y)\,dy\,dB_j(z),
\\
K''_g(x,s,v) &= \sum_{j=1}^\infty \lambda_j^{1/2} \int_0^v \int_D \bigl(G(x,s;y,z)-G(x,v;y,z)\bigr)g\left(u^N(y,z)\right)e_j(y)\,dy\,dB_j(z).
\end{align*}
By Ito's isometry,
\begin{align}
\E\norm{K'_g(\cdot,s,v)}_2^2 &= \E\int_D\left(\sum_{j=1}^\infty \lambda_j^{1/2} \int_v^s \int_D G(x,s;y,z)g\left(u^N(y,z)\right)e_j(y)\,dy\,dB_j(z)\right)^2dx
\notag
\\
&=\E\int_D\sum_{j=1}^\infty \lambda_j \int_v^s\left( \int_D G(x,s;y,z)g\left(u^N(y,z)\right)e_j(y)\,dy\right)^2dz\,dx
\notag
\\
&\le\left(\sum_{j=1}^\infty \lambda_j\norm{e_j}_\infty\right)\E\int_v^s\int_D \left( \int_D \abs{G(x,s;y,z)g\left(u^N(y,z)\right)}\,dy\right)^2dx\,dz.
\label{eq:Kg'1}
\end{align}
Then using the assumption \ref{(A5)} and bounding the inner integrals similarly to \eqref{eq:If2}, we arrive at
\[
\E\norm{K'_g(\cdot,s,v)}_2^2
\le C (s-v)\left(1+\E\norm{u^N}_{\alpha,2,s}^2 \right).
\]
Therefore,
\begin{align*}
\MoveEqLeft
\int_0^t  \int_0^s \E \norm{K'_g(\cdot,s,v)}_2 (s-v)^{-3/2-\alpha}\,dv \,ds
\\
&\le C\int_0^t  \left(1+\E\norm{u^N}_{\alpha,2,s}^2\right) \int_0^s(s-v)^{-1/2-\alpha}\,dv \,ds
\le C \left(1+ \int_0^t\E\left[\norm{u^N}_{\alpha,2,s}^2\right] ds\right).
\end{align*}

Similarly to \eqref{eq:Kg'1},
\[
\E\norm{K''_g(\cdot,s,v)}_2^2
\le C\E\int_0^v\int_D \left( \int_D \abs{G(x,s;y,z)-G(x,v;y,z)}\abs{g\left(u^N(y,z)\right)}\,dy\right)^2dx\,dz.
\]
One can bound the integral in the right-hand side in the same way as $K''_f$ in \eqref{eq:Kf''}:
\[
\E\norm{K''_g(\cdot,s,v)}_2^2
\le
C\E(s-v)^{\delta}\int_0^v (v-z)^{-\delta}\left(1+\norm{u^N(\cdot,z)}_2^2\right)dz.
\]
Further,
\begin{align*}
\E\norm{K''_g(\cdot,s,v)}_2^2
&\le
C(s-v)^{\delta}\E\left(1+\sup_{z\in[0,s]}\norm{u^N(\cdot,z)}_2^2\right)\int_0^v (v-z)^{-\delta}dz
\\
&\le
C(s-v)^{\delta}\left(1+\E\norm{u^N}_{\alpha,2,s}^2\right)
\end{align*}
Hence, choosing $\delta>\frac12+\alpha$, we obtain
\begin{align*}
\MoveEqLeft
\int_0^t  \int_0^s \E \norm{K''_g(\cdot,s,v)}_2 (s-v)^{-3/2-\alpha}\,dv \,ds
\\
&\le C\int_0^t  \left(1+\E\norm{u^N}_{\alpha,2,s}^2\right) \int_0^s(s-v)^{\delta-\frac32-\alpha}\,dv \,ds
\le C \left(1+ \int_0^t\E\left[\norm{u^N}_{\alpha,2,s}^2\right] ds\right).
\end{align*}
Thus
\begin{equation}\label{eq:l2c}
\E J_g^N(t) \le C \left(1+ \int_0^t\E\left[\norm{u^N}_{\alpha,2,s}^2\right] ds\right).
\end{equation}

Combining \eqref{eq:l2a}, \eqref{eq:l2b}, and \eqref{eq:l2c}, we get
\[
\E\left[\norm{u^N}_{\alpha,2,t}^2\right]
\le C_N\left(1+\int_0^t\E\left[\norm{u^N}_{\alpha,2,s}^2\right]ds\right),
\]
and the proof follows from Gronwall's lemma.
\end{proof}

\begin{lemma}\label{l:3}
Let the assumptions of Theorem~\ref{th:mild} hold.
Then
\[
\E\norm{u_{N,n}-u_{N,m}}_{\alpha,2,T}^2\ind_{A_T^{N,R}} \le C_{N,R} \E\left(\sum_{j=1}^\infty\mu_j^{1/2}\norm{B^{H,N,n}_j-B^{H,N,m}_j}_{\alpha,0;T}\right)^2
\]
where
$A_T^{N,R}=\set{\norm{u_{N,n}}_{\alpha,2,T}\le R, \norm{u_{N,m}}_{\alpha,2,T}\le R}$,
$B^{H,N,n}_j$ and $u_{N,n}$ are defined by \eqref{eq:approx-B}--\eqref{eq:approx-u}.
\end{lemma}

\begin{proof}
The proof will be similar to that of Lemmas~\ref{l:1}--\ref{l:2}, so we omit some details.
Denote
$\ind_t=\ind_{A_t^{N,R}}$,
$u=u_{N,n}$, $\tilde u=u_{N,m}$,
$Z_j=B^{H,N,n}_j$, $\widetilde Z_j=B^{H,N,m}_j$,
$\Delta Z_j=Z_j-\widetilde Z_j$,
$\eta_{\alpha,H,t}=\sum_{j=1}^\infty\mu_j^{1/2}\norm{\Delta Z_j}_{\alpha,0;T}$.
We have
\begin{align*}
\norm{u-\tilde u}_{\alpha,2,t}^2
&\le C\sup_{s\in[0,t]}\left(\norm{I_{\Delta f}(\cdot,s)}_2^2 + \norm{I_{\Delta g}(\cdot,s)}_2^2 + \norm{I_{\Delta h}(\cdot,s)}_2^2
+\norm{I_{\Delta Z}(\cdot,s)}_2^2\right)
\\
&\quad + J_{\Delta f}(t) + J_{\Delta g}(t) + J_{\Delta h}(t)+ J_{\Delta Z}(t),
\end{align*}
where
\begin{align*}
I_{\Delta f}(x,t) &\coloneqq \int_0^t\int_D G(x,t;y,s)\bigl(f(u(y,s))-f(\tilde u(y,s))\bigr)\,dy\,ds,
\\
I_{\Delta g}(x,t) &\coloneqq \sum_{j=1}^\infty \lambda_j^{1/2} \int_0^t \int_D G(x,t;y,s) \bigl(g(u(y,s))-g(\tilde u(y,s))\bigr)e_j(y)\,dy\,dB_j(s),
\\
I_{\Delta h}(x,t) &\coloneqq \sum_{j=1}^\infty \mu_j^{1/2} \int_0^t \int_D G(x,t;y,s)\bigl(h(u(y,s))-h(\tilde u(y,s))\bigr)e_j(y)\,dy\,dZ_j(s),
\\
I_{\Delta Z}(x,t) &\coloneqq \sum_{j=1}^\infty \mu_j^{1/2} \int_0^t \int_D G(x,t;y,s)h(\tilde u(y,s))e_j(y)\,dy\,d
\bigl(\Delta Z_j(s)\bigr),
\\
J_a(t) &\coloneqq \int_0^t \left( \int_0^s \frac{\norm{I_a(\cdot,s)-I_a(\cdot,v)}_2}{(s-v)^{\alpha+1}}\,dv \right )^2 ds,
\quad a\in\set{\Delta f, \Delta g, \Delta h, \Delta Z}.
\end{align*}

Similarly to \eqref{eq:Ih0},
\begin{align*}
\norm{I_{\Delta Z}(\cdot,t)}_2 &\le C \eta_{\alpha,H,t} \left(1 + \int_0^t\frac{\norm{\tilde u(\cdot,s)}_2}{s^\alpha}\,ds
+ \int_0^t\int_0^s \frac{\norm{\tilde u(\cdot,s)-\tilde u(\cdot,v)}_2}{(s-v)^{\alpha+1}}\,dv\,ds
\right.
\\
&\quad+\left. \int_0^t (t-s)^{-\delta/2} \int_0^s (s-v)^{\delta/2-\alpha-1} \left(1 + \norm{\tilde u(\cdot,v)}_2\right) \,dv\,ds \right),
\end{align*}
for any $\delta\in\left(\frac{d}{d+2},1\right)$.
Arguing as in the proof of Lemma~\ref{l:1} for the term $I_h$, we obtain
\begin{align*}
\norm{I_{\Delta Z}(\cdot,t)}_2^2 &\le C \eta_{\alpha,H,t}^2 \left(1 + \int_0^t\norm{\tilde u(\cdot,s)}_2^2\,ds
+ \int_0^t\left(\int_0^s \frac{\norm{\tilde u(\cdot,s)-\tilde u(\cdot,v)}_2}{(s-v)^{\alpha+1}}\,dv\right)^2ds\right)
\\
&\le C \eta_{\alpha,H,t}^2 \left(1 + \norm{u}_{\alpha,2,t}^2\right).
\end{align*}
Therefore
\[
\norm{I_{\Delta Z}(\cdot,t)}_2^2\ind_t
\le C R^2 \eta_{\alpha,H,t}^2.
\]
The term $J_{\Delta Z}$ can be bounded analogously to \eqref{eq:Jh}:
\[
J_{\Delta Z}(t)\ind_t\le C\eta_{\alpha,H,t}^2\left(1+\int_0^t\norm{\tilde u}_{\alpha,2,s}^2ds\right)\ind_t
\le C R^2 \eta_{\alpha,H,t}^2.
\]
The terms $I_{\Delta f}$ and  $J_{\Delta f}$ can be estimated in the same way as the terms $I_{f}$ and  $J_{f}$ in the proof of Lemma~\ref{l:1}, using the Lipschitz condition \ref{(A2)}
instead of the inequality $\abs{f(u)}\le C(1+\abs{u})$.
This leads to the bounds
\[
\sup_{s\in[0,t]}\norm{I_{\Delta f}(\cdot,t)}_2^2
\le C\int_0^t\norm{u-\tilde u}_{\alpha,2,s}^2ds
\quad\text{and}\quad
J_{\Delta f}(t) \le C\int_0^t\norm{u-\tilde u}_{\alpha,2,s}^2ds.
\]
In order to estimate $I_{\Delta h}$, we can use the same arguments as for $I_{h}$ in the proof of Lemma~\ref{l:1}, using the bounds
\[
\sup_{j\in\N}\norm{a_{j,t}(u)(\cdot,s)-a_{j,t}(\tilde u)(\cdot,s)}_2 \le C (\norm{u(\cdot,s)-\tilde u(\cdot,s)}_2)
\]
and
\begin{align}
\MoveEqLeft
\sup_{j\in\N}\norm{a_{j,t}(u)(\cdot,s)-a_{j,t}(\tilde u)(\cdot,s)-a_{j,t}(u)(\cdot,v)+a_{j,t}(\tilde u)(\cdot,v)}_2
\notag
\\*
&\le C(t-s)^{-\delta} (s-v)^{\delta} \norm{u(\cdot,v)-\tilde u(\cdot,v)}_2
\notag
\\*
&\quad+C\norm{u(\cdot,s)-\tilde u(\cdot,s) -u(\cdot,v)
+ \tilde u(\cdot,v)}_2,
\label{eq:ddf-bound}
\end{align}
instead of \eqref{eq:f-bound} and \eqref{eq:df-bound} respectively.
These bounds can be established similarly to \eqref{eq:f-bound} and \eqref{eq:df-bound}, see \cite[Lemma~3.3]{SSV09} for their proofs.
Mention that for \eqref{eq:ddf-bound} we need the assumption that $h$ is an affine function.
We will obtain
\[
\sup_{s\in[0,t]}\norm{I_{\Delta h}(\cdot,t)}_2^2
\le C_N\int_0^t\norm{u-\tilde u}_{\alpha,2,s}^2ds.
\]

Finally, the bound
\[
J_{\Delta h}(t)\le C_N\int_0^t\norm{u-\tilde u}_{\alpha,2,s}^2ds.
\]
can be proved similarly to \eqref{eq:Jh}, inequalities \eqref{eq:f*-bound} and \eqref{eq:df*-bound} are replaced by (3.37) and (3.38) from \cite{SSV09} respectively.
Note that for this term we need to choose $\delta\in\left(\frac{d}{d+2},1-2\alpha\right)$,
this leads to the restriction $\alpha<\frac{1}{d+2}$.
We refer to the proof of part (b) of Theorem 2.3 in \cite{SSV09} for the details on the estimation of $I_{\Delta h}$ and $J_{\Delta h}$.

Thus, we see that
\begin{align*}
\norm{u-\tilde u}_{\alpha,2,t}^2
\!\ind_t
&\le C_{N,R}\!\left[\eta_{\alpha,H,t}^2 + \int_0^t\norm{u-\tilde u}_{\alpha,2,s}^2\ind_s ds + \sup_{s\in[0,t]}\norm{I_{\Delta g}(\cdot,t)}_2^2\ind_t+
J_{\Delta g}(t)\ind_t\right]\!.
\end{align*}

Similarly to \eqref{eq:l2b} (using the Lipschitz continuity of $g$ instead of the linear growth condition), we get
\begin{align*}
\E\left[ \sup_{s\in[0,t]}\norm{I_{\Delta g}(\cdot,s)}_2^2 \ind_t\right]
&\le C  \int_0^t\E\left[\norm{u-\tilde u}_{\alpha,2,s}^2\ind_s\right] ds,
\\
\E \left[J_{\Delta g}(t)\ind_t\right]
&\le C \int_0^t\E\left[\norm{u-\tilde u}_{\alpha,2,s}^2\ind_s\right] ds.
\end{align*}

Then
\begin{align*}
\E\left[\norm{u-\tilde u}_{\alpha,2,t}^2\ind_t\right]
&\le C_{N,R}\left(\E \eta_{\alpha,H,t}^2
+ \int_0^t \E \left[\norm{u-\tilde u}_{\alpha,2,s}^2\ind_s\right] ds \right),
\end{align*}
and the result follows from Gronwall's lemma.
\end{proof}

\begin{remark}\label{rem:l3}
It follows from the above proof that the statement of Lemma~\ref{l:3} remains true, if we replace the sequences $\set{B^{H,N,n}_j,j\in\N}$ and $\set{B^{H,N,m}_j,j\in\N}$ by other two sequences $\set{Z_j,j\in\N}$ and $\set{\widetilde Z_j,j\in\N}$ of independent identically distributed H\"older continuous processes, satisfying the assumptions
\[
\sum_{j=1}^\infty\mu_j^{1/2}\norm{Z_j}_{\alpha,0;T}\le N
\quad\text{and}\quad
\sum_{j=1}^\infty\mu_j^{1/2}\norm{\widetilde Z_j}_{\alpha,0;T}\le N.
\]
More precisely, if $u$ and $\tilde u$ are two processes, satisfying \eqref{eq:mild} with $B^H_j$ replaced by $Z_j$ and $\widetilde Z_j$ respectively, then we have
\[
\E\norm{u-\tilde u}_{\alpha,2,T}^2\ind_{A_T^{R}} \le C_{N,R} \E\left(\sum_{j=1}^\infty\mu_j^{1/2}\norm{Z_j-\widetilde Z_j}_{\alpha,0;T}\right)^2,
\]
where
$A_T^{R}=\set{\norm{u}_{\alpha,2,T}\le R, \norm{\tilde u}_{\alpha,2,T}\le R}$.
For example, we may consider the processes
$Z_j=B^{H,N,n}_j$ and $\widetilde Z_j=B^{H,N}_j$ from the proof of Theorem \ref{th:mild} and the corresponding solutions $u=u_{N,n}$ and $\tilde u=u_N$.
\end{remark}

\textbf{Acknowledgement.}
Yu. Mishura is thankful to M. Dozzi for the fruitful discussion of the topic during her visit to the University of Lorraine.

\end{document}